\documentclass[11pt,reqno]{amsart}
\usepackage{latexsym}
\RequirePackage{amsmath}   \RequirePackage{amssymb}
\RequirePackage{amsthm}   \RequirePackage{epsfig}
\theoremstyle{plain}
\newtheorem{theorem}{Theorem}
\newtheorem{corollary}[theorem]{Corollary}
\newtheorem{lemma}[theorem]{Lemma}  
\newtheorem{proposition}[theorem]{Proposition}  
\newtheorem*{corollary*}{Corollary}
\newtheorem{theoremO}{Theorem}

\newtheorem*{conjecture*}{Conjecture}

\theoremstyle{definition}

\theoremstyle{remark}
\newtheorem{remark}{Remark}
\newtheorem{example}[remark]{Example}

\newcommand{\SC}{{\mathbb C}}  \newcommand{\SD}{{\mathbb D}}  
\newcommand {\SR}{{\mathbb R}}  \newcommand{\ST}{{\mathbb T}}  
\newcommand{\al}{\alpha}  \newcommand{\ga}{\gamma}  
  \newcommand{\ve}{\varepsilon}  \newcommand{\ze}{\zeta}
    \newcommand{\la}{\lambda}
  \newcommand{\vp}{\varphi}  \newcommand{\om}{\omega}
\newcommand{\Om}{\Omega}

\newcommand{\be}{\begin{equation}}
\newcommand{\ee}{\end{equation}}
\newcommand{\bea}{\begin{eqnarray}}
\newcommand{\eea}{\end{eqnarray}}

\begin{document}

\title{A class of Weierstrass-Enneper lifts of harmonic mappings}  

\author{Martin Chuaqui} 
\address{Facultad de Matem\'aticas, Pontificia Universidad Cat\'olica de Chile, Casilla 306, Santiago 22, Chile.} \email{mchuaqui@mat.uc.cl}

\author{Iason Efraimidis}
\address{Facultad de Matem\'aticas, Pontificia Universidad Cat\'olica de Chile, Casilla 306, Santiago 22, Chile.} \email{iason.efraimidis@mat.uc.cl}

\subjclass[2010]{30C99, 30C62, 53A10}
\keywords{Harmonic mapping, Schwarzian derivative, univalence criterion}

\maketitle

\begin{abstract}
We introduce a class of Weierstrass-Enneper lifts of harmonic mappings which satisfy a criterion for univalence introduced by Duren, Osgood and the first author in [J. Geom. Anal. (2007)]. 
\end{abstract}

\section{Introduction} 
Let $f$ be an analytic locally univalent function in the unit disk $\SD$ and 
$$ 
Sf \, = \, \left( f'' /f' \right)' -\tfrac{1}{2} \left( f'' /f'  \right)^2   
$$
be the \textit{Schwarzian derivative} of $f$. Nehari \cite{Ne54} proved that the bound 
\be \label{Neh-p}
|Sf(z)| \leq 2 \, p(|z|) \, , \qquad z\in\SD,
\ee
implies the global univalence of $f$ and thus unified some specific instances of this theorem that were then known, for instance, the ones corresponding to the functions $p(x)$ given by 
$$
\frac{1}{(1-x^2)^2}\, , \qquad \frac{2}{1-x^2} \qquad \text{and} \qquad \frac{\pi^2}{4}\,. 
$$
In general, $p$ is a positive, continuous, even function defined on $(-1,1)$ with the properties that $(1-x^2)^2p(x)$ is nonincreasing on the interval $(0,1)$ and that the differential equation $u'' + pu = 0$ is disconjugate. The latter means that every nontrivial solution of this equation has at most one zero in $(-1,1)$ or, equivalently, that some solution has no zeros in $(-1,1)$. We refer to such functions $p$ as \textit{Nehari functions}. Nehari's theorem has been generalized in the context of harmonic mappings, as we shall shortly see. 

A complex-valued \emph{harmonic mapping} $f$ in a simply connected domain $\Omega\subset\SC$ has a canonical decomposition $f=h+\overline{g}$, where $h$ and $g$ are analytic in $\Omega$ and $g(z_0)=0$ for some specified $z_0\in\Omega$. The mapping $f$ is locally univalent if and only if its Jacobian $J_f=|h'|^2-|g'|^2$ does not vanish, and is said to be \textit{orientation-preserving} if its dilatation $\omega=g'/h'$ satisfies $|\omega|<1$ in $\Omega$. 

According to the Weierstrass-Enneper formulas a harmonic mapping $f=h+\overline{g}$ with $|h'|+|g'|\neq0$ can be \textit{lifted} locally to a minimal surface described by conformal  parameters if and only if its dilatation has the form $\omega=q^2$ for some meromorphic function $q$. The lifted mapping $\widetilde f$ from $\Omega$ to the minimal surface has Cartesian coordinates $\widetilde f\, = \, \big(U,V,W\big)$ given by 
$$
U(z) = {\rm Re}\, f(z) \, , \qquad V(z) = {\rm Im}\, f(z) \, , \qquad W(z) = 2\,{\rm Im}\, \left(\int_{z_0}^z \sqrt{h'(\ze)g'(\ze)} d\ze\right), 
$$ 
for $z\in\Omega$. The first fundamental form of the surface is $ds^2=e^{2\sigma}|dz|^2$, where the conformal factor is
$$
e^\sigma \, = \, |h'| + |g'|.
$$
The \textit{Gauss curvature} of the surface at a point $\widetilde f(z)$ for which $h'(z)\neq 0$ is
$$
K \, = \, -e^{-2\sigma} \Delta \sigma \, = \, - \frac{4|q'|^2}{|h'|^2 (1+|q|^2)^4}\,,
$$
where $\Delta$ is the Laplace operator. For harmonic mappings that admit a lift the \textit{Schwarzian derivative} was introduced in \cite{CDO03} as $Sf=2(\sigma_{zz}-\sigma_z^2)$. When $h'(z)\neq 0$ this produces the expression
\be \label{h-Schw}
Sf \, = \, Sh + \frac{2\overline{q}}{1+|q|^2} \left(q'' - q' \frac{h''}{h'} \right) - 4 \left( \frac{q' \overline{q}}{1+|q|^2} \right)^2. 
\ee
See Chapters 9 and 10 in Duren's book \cite{Du3} for further information on this topic. 

For the case when $\Omega=\SD$, the following generalization of Nehari's univalence criterion \eqref{Neh-p} was proved in \cite{CDO07}. The relevant quantity here is 
\be \label{def-Phi}
\Phi_f(z) \, = \, |Sf(z)| + e^{2\sigma(z)}|K\big(\widetilde f(z)\big)|. 
\ee
In the case when $f$ is analytic this expression reduces to $|Sf(z)|$, since $K=0$. 

\begin{theoremO}[\cite{CDO07}]
Let $f=h+\overline{g}$ be a harmonic mapping of $\SD$ with $|h'|+|g'|\neq0$ and dilatation $g'/h'=q^2$ for some meromorphic function $q$. If 
\be \label{lift-p}
\Phi_f(z) \, \leq \, 2 p(|z|), \qquad z\in\SD,
\ee
for some Nehari function $p$ then $\widetilde f$ is univalent in $\SD$ . 
\end{theoremO}

Inequality \eqref{lift-p} is the core of the present article since our objective is to find a wide class of harmonic lifts that satisfy it. Our first theorem assumes that the analytic part of a harmonic mapping has ``small'' Schwarzian derivative and finds conditions on the dilatation which ensure that \eqref{lift-p} holds. 

\begin{theorem} \label{h-thm}
Suppose $h$ and $q$ are analytic in $\SD$ and such that $h'\neq0$, 
$$
|Sh(z)| \leq \frac{2\, s}{(1-|z|^2)^2}, \quad z\in\SD
$$
for some $s \in [0,1]$ and that $\om=q^2$ satisfies 
$$
\rho(s,t,R) \leq |\om(z)| \leq R, \qquad z\in\SD,   
$$
for some $t\in[s,1]$ and $R>0$, where the function $\rho$ is given by 
$$
\rho(s,t,R) \, = \, \max \left\{ 0, \; R - \frac{(R+1)(t-s)}{t-s+2( 1 + \sqrt{1+s} )} \right\}, \quad 0 \leq  s  \leq t \leq1, \; R>0. 
$$
Then the lift $\widetilde f$ of the mapping $f=h+\overline{g}$ with dilatation $\om=g'/h'$ satisfies  
\be \label{lift-t-thm}
\Phi_f(z)  \, \leq \, \frac{2 \, t}{(1-|z|^2)^2}, \qquad z\in\SD
\ee
and, in particular, $\widetilde f$ is univalent. 
\end{theorem}

Note that the change of formula for the function $\rho$ in Theorem~\ref{h-thm} occurs at
$$
R_0 \,=  \, \frac{t-s}{2( 1 + \sqrt{1+s} )}. 
$$
Therefore, the hypothesis on $\om$ corresponds to a disk for $R\in (0,R_0]$ and to an annulus for $R> R_0$, and passes from one to the other continuously. In Proposition~\ref{h-prop} an explicit example will be given for which $R$ is ``large'' and \eqref{lift-p} fails to hold, and will thus justify the lower-bound hypothesis. 


Another of our theorems, Theorem~\ref{vp-thm}, explores the situation when $f=h+\overline{g}$ is given in terms of its dilatation $\om= g'/h'$ and the equation $h-g=\vp$, for some analytic function $\vp$. According to a theorem of Clunie and Sheil-Small $f$ is univalent and convex in the horizontal direction if and only if $\vp$ has the same properties (see \cite[\S3.4]{Du3}), in which case we say that $f$ is the \emph{shear} of $\vp$. However, in Theorem~\ref{vp-thm} the function $\vp$ is only assumed to be locally univalent and have ``small'' Schwarzian derivative. Then two (non-overlapping) conditions on the dilatation are given, ensuring that inequality \eqref{lift-p} -in particular \eqref{lift-t-thm}- is satisfied. One of the two conditions provides a two-variable function $\eta$ with the property that if 
$$
|S\vp(z)| \leq \frac{2 \, s}{(1-|z|^2)^2} \qquad \text{and} \qquad |\om(z)| \leq \eta(s,t), \qquad z\in\SD,  
$$
for some $0 \leq s \leq t \leq1$ then $\widetilde f$ satisfies \eqref{lift-t-thm}. 

In the recent article \cite{CDO18}, inequality \eqref{lift-t-thm} for $t<1$ was shown to imply that $\widetilde f$ has a quasiconformal extension to $\overline{\SR^3}$. This was then used to prove that if the dilatation of $f$ is assumed to be sufficiently small then the planar mapping $f$ is univalent and admits a quasiconformal extension to $\SC$. 
   
\begin{theoremO} [\cite{CDO18}] \label{plane-qc-ext}
Suppose $f=h+\overline{g}$ is a locally injective harmonic mapping of $\SD$ whose lift $\widetilde f$ satisfies \eqref{lift-t-thm} for a $t<1$ and whose dilatation $\om$ satisfies 
$$
\sup_{z\in\SD} |\om(z)| \, < \, \left( \frac{1-\sqrt{t}}{1+\sqrt{t}} \right)^2. 
$$
Then $f$ is injective and has a quasiconformal extension to $\SC$. 
\end{theoremO}

An explicit formula was given in \cite{CDO18} for the extension of $f$ to $\SC$, but here we omit it for simplicity. In Section~\ref{sect-shear} we will combine Theorem~\ref{plane-qc-ext} with Theorem~\ref{vp-thm} in order to prove the following proposition. 

\begin{corollary} \label{cor-vp}
There exists a positive function $\hat t$ on $[0,1)$ such that if $\vp$ and $q$ are analytic in $\SD$, $\vp'\neq0$ and 
\be \label{S-qd-vp}
|S\vp(z)| \leq  \frac{2 \, s }{(1-|z|^2)^2}, \quad z\in\SD
\ee
for some $s \in [0,1]$ and $\om=q^2$ satisfies 
$$
|\om(z)| \leq \eta(s,t), \qquad z\in\SD, 
$$
for some $t \in [s,\hat t(s) )$  then the mapping $f=h+\overline{g}$, with $h-g=\vp$ and dilatation $\om=g'/h'$, is univalent and has quasiconformal extension to $\SC$.  
\end{corollary}

A similar corollary can be drawn combining Theorem~\ref{h-thm} with Theorem~\ref{plane-qc-ext}, but we will not explicitly state it here since it lacks the geometric implications that, as we will now see, Corollary~\ref{cor-vp} has. Its proof would be similar to the proof of Corollary~\ref{cor-vp}.

According to a well-known theorem of Ahlfors and Weill, inequality \eqref{S-qd-vp} for $s<1$ implies that $\Om =\vp(\SD)$ is a quasidisk (see \S 5.6 in \cite{P92}). Therefore one may ask if condition \eqref{S-qd-vp} can be weakened and if, in general, for any quasidisk $\Om \subset \SC$ there exists an $\ve=\ve(\Om)>0$ such that any shear of a conformal mapping $\vp$ of $\SD$ onto $\Om$ whose dilatation $\om$ satisfies $\|\om \|_\infty \leq \ve$ is univalent. In looser terms: Is it always possible to shear a quasidisk? 

In fact, this is true in even greater generality. A domain $\Om \subset \SC$ is said to satisfy an interior \emph{chord-arc} condition if there exists a constant $M=M(\Om)>0$ such that any two points $z,w\in\Om$ satisfy
\be \label{def-M}
\ell(z,w) \, := \, \inf\left\{ \int_\gamma |dz| \, : \,   \gamma\subset\Om, \, z,w\in\gamma\right\} \, \leq \, M \, |z-w|, 
\ee
where the $\gamma$'s are arcs. This readily implies that $\Om$ is bounded by Jordan curves. If its boundary is piecewise smooth then the chord-arc condition is equivalent to the absence of inward-pointing cusps. The importance of this condition lies in the fact that all quasidisks satisfy it (see \cite{ABG} or \cite[\S5.4, exercise 4]{P92}). 

Let $\Om \subset \SC$ be a simply connected domain and $\vp$ a conformal mapping of $\SD$ onto $\Om$. Recall that a shear of $\vp$ in the direction $\la\in\ST$ ($=\partial \SD$) is a harmonic mapping $f=h+\overline{g}$ for which $h-\la^2 g=\vp$. The following theorem was proved in \cite{CPW}.

\begin{theoremO}[\cite{CPW}]\label{Ponnu}
If $\Om=\vp(\SD)$ satisfies the chord-arc condition \eqref{def-M} then any shear of $\vp$ whose dilatation $\om$ satisfies $|\om |\leq\ve$ in $\SD$, for some $ \ve < (2M+1)^{-1}$, is univalent. 
\end{theoremO}

An inspection of the proof reveals that both hypothesis and conclusions can be stated in terms of the (weaker) notion of \emph{directional} chord-arc condition. For a direction $\la\in\ST$ this condition is satisfied if the quantity 
$$ 
M(\la) \, = \, \sup\left\{ \frac{\ell(z,w)}{|z-w|} \, : \, z,w\in\Om, \,\, \bar\la(z-w)\in\SR \right\}
$$ 
is finite. Clearly the constant in \eqref{def-M} is $M=\sup_{\la\in\ST} M(\la)$. Therefore we have the following. 

\begin{theorem}\label{direct-thm}
If $\Om=\vp(\SD)$ and $M(\la)<\infty$ for some $\la\in\ST$ then any shear of $\vp$ in the direction $\la$ whose dilatation $\om$ satisfies $|\om |\leq\ve$ in $\SD$, for some $ \ve < \big(2M(\la)+1\big)^{-1}$, is univalent. 
\end{theorem}
 
An incorrect statement was made in \cite{CPW} as to a converse of Theorem~\ref{Ponnu}. It claimed that if there exists a constant $c>0$ such that any shear of $\vp$ in a fixed direction with dilatation $\om$ satisfying $|\om|\leq c$ is univalent then $\Om$ satisfies the chord-arc condition \eqref{def-M}. For a proof of this the authors suggested that one should argue as in \cite{CH}.

It is easy to produce domains which are convex in a given direction but do not satisfy the chord-arc condition. Consider, for example, a domain with an inward-pointing cusp whose direction (of the common tangent of the two parts of the boundary meeting at the vertex of the cusp) is parallel to the direction of convexity. Clearly the theorem of Clunie and Sheil-Small can be applied to the conformal mapping of such a domain and, moreover, any analytic $\om:\SD\to\SD$ is an admissible dilatation. Theorem~\ref{direct-thm} can also be applied here, although, since $M(\la)=1$ in the direction of convexity, it yields the non-optimal constant $1/3$.

We will prove the following converse of Theorem~\ref{Ponnu} in Section~\ref{sect-shear}. It merely yields the existence of a direction in which no shearing is possible. 

\begin{theorem} \label{converse}
If $\Om=\vp(\SD)$ does not satisfy the chord-arc condition \eqref{def-M} then there exists a direction $\la\in\ST$ such that for any $\ve>0$ there exists a non-univalent harmonic mapping $f=h+\overline{g}$ for which $h-\la^2 g = \vp$ and $|\om|<\ve$ in $\SD$. 
\end{theorem}



\section{Auxiliary lemmas}
The family of analytic locally univalent functions whose \emph{Schwarzian norm}
$$
\|Sf\| \, =\, \sup_{z\in\SD}(1-|z|^2)^2 |Sf(z)|
$$
is bounded by a fixed number is linearly invariant. The supremum of the second Taylor coefficient (in other words, the \emph{order}) of this family is computed in a well-known theorem of Pommerenke \cite[p.133]{Po64}. In view of the expression for the second coefficient of the composition of a function with a disk automorphism (see \cite[\S2.3]{Du2}), Pommerenke's theorem can easily be stated as a distortion theorem.

\begin{theoremO}[\cite{Po64}]\label{h}
If $f(z)=\sum_{n=0}^{\infty} a_n z^n$ is an analytic locally univalent function in $\SD$ and $\|Sf\| \leq 2 t$ for some $t\in[0,1]$ then 
$$
\left| \frac{f''(z)}{f'(z)}  -\frac{2 \, \bar{z}}{1-|z|^2} \right| \, \leq \, \frac{2 \sqrt{1+t}}{1-|z|^2}, \qquad z\in\SD. 
$$
In particular, it holds that 
$$
|a_2|\leq |a_1|\sqrt{1+t}. 
$$
Both inequalities are sharp for 
$$
f(z) \, = \,  \frac{1}{2\sqrt{1+t}} \left( \frac{1+z}{1-z} \right)^{\sqrt{1+t}}.
$$
\end{theoremO}

We will also make use of the following elementary lemma. 
\begin{lemma} \label{Wiener}
If $\om:\SD\to\SC$ is analytic and such that $|\om(z)|<R$ in $\SD$ for some $R>0$ then  
$$
\left| \om''(z) -\frac{2 \, \bar{z} \, \om'(z)}{1-|z|^2} \right| \, \leq \, \frac{2(R^2-|\om(z)|^2)}{R(1-|z|^2)^2}, \qquad z\in\SD. 
$$
\end{lemma}
\begin{proof}
We will apply the well-known inequality
\be \label{wie}
|a_n| \leq 1- |a_0|^2
\ee
to the function
$$
\psi(z) = \frac{1}{R} \, \om\left(\frac{\al-z}{1-\bar{\al}z} \right) = a_0 + a_1 z + a_2 z^2 + \ldots 
$$
Inequality \eqref{wie} is attributed either to Littlewood or F.W. Wiener and can be found in \cite{Boh} and \cite[p.172]{Ne75}. We compute
$$
a_0= \frac{\om(\al)}{R}, \qquad a_2 =  \frac{(1-|\al|^2)^2}{2 R} \left( \om''(\al) -\frac{2\overline{\al}\om'(\al)}{1-|\al|^2} \right)
$$
and directly deduce the desired inequality from \eqref{wie}, for $n=2$. 
\end{proof}

\section{Conditions on the Schwarzian of the analytic part}
We begin this section with the proof of our first theorem. 

\begin{proof}[Proof of Theorem~\ref{h-thm}.]
We set $R=r^2$ so that $|q|\leq r$, since $\om=q^2$ and $|\om|\leq R$. Note that the second term in \eqref{def-Phi} is given by
\be \label{prelim-curv}
e^{2\sigma(z)}|K\big(\widetilde f(z)\big)| \, = \,\frac{4|q'(z)|^2}{(1+|q(z)|^2)^2}, \qquad z\in\SD.
\ee
This will be grouped with the module of the third term in the Schwarzian derivative \eqref{h-Schw} as follows
\be \label{curv}
e^{2\sigma}|K (\widetilde f )| + \frac{4 |q' \overline{q}|^2}{(1+|q|^2)^2} \, = \, \frac{4|q'|^2}{1+|q|^2}.
\ee
We also add and subtract the term $\frac{2 \, \bar{z} \, q'}{1-|z|^2}$ and compute 
$$
\Phi_f(z) \, \leq \, |Sh| + \frac{2|q|}{1+|q|^2} \left( \left| q'' -\frac{2 \, \bar{z} \, q'}{1-|z|^2} \right| + |q'|\left| \frac{h''}{h'}  -\frac{2 \, \bar{z}}{1-|z|^2} \right|\right) +\frac{4|q'|^2}{1+|q|^2}  . 
$$
An application of the Schwarz-Pick lemma to the function $q/r$ gives
\be \label{Sch-Pick}
|q'(z)|\, \leq \, \frac{r^2-|q(z)|^2}{r(1-|z|^2)}, \quad z\in\SD. 
\ee
Together with Lemmas~\ref{h} and~\ref{Wiener}, and the assumption that $|\om|\geq \rho$, this yields 
\begin{align*}
& (1-|z|^2)^2 \, \Phi_f(z) \\  
& \leq \,  2 s  +  \frac{4|q(z)|(r^2-|q(z)|^2)}{r(1+|q(z)|^2)} \left( 1+ \sqrt{1+s} \right) + \frac{4(r^2-|q(z)|^2)^2}{r^2(1+|q(z)|^2)} \\
& = \,  2 s + \frac{4(r^2-|q(z)|^2)}{r^2(1+|q(z)|^2)} \left(  r|q(z)| (1+ \sqrt{1+s} )  + r^2-|q(z)|^2 \right) \\ 
& \leq \, 2 s +  \frac{4(r^2-\rho)}{r^2(1+\rho)} \, r^2 (1+\sqrt{1+s}) \\
& = \, 2 s + 
\frac{4(R-\rho)}{1+\rho} \, (1+\sqrt{1+s}), 
\end{align*} 
where the last inequality follows from the observation that $\frac{r^2-|q|^2}{1+|q|^2}$ decreases with $|q|$, while the quadratic polynomial in the parenthesis increases with $|q|$. Now, this is less than or equal to $2t$ if and only if 
$$
\rho \, \geq \, R - \frac{(R+1)(t-s)}{t-s+2( 1 + \sqrt{1+s} )}.
$$
The proof is complete. 
\end{proof}
 
The following proposition illustrates that some conditions on the dilatation such as the ones given in Theorem~\ref{h-thm} are, indeed, necessary. In particular, for a specific dilatation whose image is a disk it shows that the radius can not be too large. Therefore, it justifies the hypothesis of ``something'' being removed from its interior. 


\begin{proposition} \label{h-prop}
For every $s$ and $t$ in $[0,1]$, such that $s\leq t$, and every $R>\frac{t-s}{2}$ there exists a harmonic mapping $f=h+\overline{g}$ for which $h'\neq0$ in $\SD$,
$$
\|Sh\| = 2s , \qquad \om(z)=R \, z^2
$$
and 
$$
\Phi_f(0) \, = \, 2s +4R \, > \, 2t. 
$$
\end{proposition}
\begin{proof}
Let
$$
h(z) = \left( \frac{1+z}{1-z} \right)^\al, \qquad \al = \sqrt{1-s}, 
$$
for which we easily compute that $Sh(z) = 2s(1-z^2)^{-2}$. We write $\om=q^2$, with $q(z)=rz$ and $R=r^2$, and compute
$$
Sf(z) \, = \, Sh(z) -\frac{2r^2 \overline{z}}{1+r^2|z|^2}\frac{h''(z)}{h'(z)} - 4 \left( \frac{r^2\overline{z}}{1+r^2|z|^2} \right)^2. 
$$
In view of \eqref{prelim-curv}, we find that 
$$
e^{2\sigma(z)}|K\big(\widetilde f(z)\big)| \, = \,\frac{4 r^2}{(1+r^2|z|^2)^2}. 
$$
The statement is now evident. 
\end{proof}

It is interesting to ask whether for fixed $s$ and $t$, with $s\leq t$, and for 
$$
\frac{t-s}{2( 1 + \sqrt{1+s} )} \, < \, R \, \leq \, \frac{t-s}{2}, 
$$
the function $\rho(s,t,R)$ in Theorem~\ref{h-thm} can be improved (lowered). What is the largest $R=R(s,t)$ for which the condition on the image of $\omega$ -ensuring that \eqref{lift-t-thm} holds- corresponds to a simply connected domain?

One may also ask if a lower bound on the dilatation along with a bound on the Schwarzian of $h$ can imply the criterion \eqref{lift-t-thm}. The following example shows that this is not possible. 
\begin{example}
Let $f=h+\bar{g}$, with $\om=g'/h'$, be given by
$$
h(z)=\frac{z}{1+z} \qquad \text{and} \qquad \om(z)=\frac{2R}{1-z}, \qquad z\in\SD,
$$
for some $R>0$. It's easy to see that $|\om|\geq {\rm Re}\,\om > R$. Since $h$ is a M\"obius transformation we have that $Sh\equiv0$. A straightforward computation shows that
$$
\Phi_f(0) \, = \, \frac{R(10R+9)}{(2R+1)^2}, 
$$
which exceeds 2 if and only if $R>\frac{\sqrt{17}-1}{4}\approx 0.78$, in which case criterion \eqref{lift-t-thm} fails.
\end{example}
 

However, it is possible for a harmonic mapping to satisfy the criterion \eqref{lift-t-thm} and have a dilatation with infinite range which, in particular, includes the real positive semi-axis.


\begin{example}
Let $f=h+\bar{g}$, with $\om=q^2=g'/h'$, be given by
$$
h(z)=z \qquad \text{and} \qquad q(z)=a \log\frac{1}{1-z}, \qquad z\in\SD,
$$
for some $a>0$. Since $Sh\equiv0$ and $h''/h'\equiv0$, we get from \eqref{h-Schw}, \eqref{def-Phi} and \eqref{curv} that 
$$
\Phi_f(z) \, \leq \, \frac{2|\overline{q(z)} q''(z)|}{1+|q(z)|^2} + \frac{4|q'(z)|^2}{1+|q(z)|^2} \, =  \, \frac{2a (2a+|q(z)|)}{|1-z|^2(1+|q(z)|^2)}  .
$$
We set $y=|q(z)|\geq 0$ and note that 
$$
\frac{2a+y}{1+y^2} \, \leq \, \frac{4a+1}{2}
$$
is equivalent to $(1-y)^2+4ay^2\geq 0$, which is true. Therefore, we have that
$$
(1-|z|^2)^2\Phi_f(z) \, \leq \,  8a \, \frac{2a+y}{1+y^2}  \, \leq \, 4a(4a+1). 
$$
This is less than or equal to $2t$ if and only if 
$$
a\leq \frac{\sqrt{1+8t}-1}{8},
$$
in which case criterion \eqref{lift-t-thm} is satisfied.
\end{example}



\section{Shears of univalent functions} \label{sect-shear}

We now turn to harmonic mappings $f=h+\overline{g}$ that are given in terms of their dilatation $\om= g'/h'$ and the equation $h-g=\vp$, for some analytic locally univalent function $\vp$. The Schwarzian derivate of $f$ was shown in \cite{CDO03} to be
\begin{align*}
Sf  \, = \, & \, S\vp + \frac{2\big(q'^{\,2}+(1-q^2)qq''\big)}{(1-q^2)^2} - \frac{2qq'}{1-q^2}\frac{\vp''}{\vp'} \\
& + \frac{2\overline{q}}{1+|q|^2} \left[ q'' -q' \left( \frac{\vp''}{\vp'} +\frac{2qq'}{1-q^2} \right)\right] -4 \left( \frac{q'\overline{q}}{1+|q|^2} \right)^2. 
\end{align*}
We rewrite this as
\begin{align} \label{Schw-vp}
Sf  \,  = \, & \, S\vp + 2   \left( \frac{q}{1-q^2} + \frac{\overline{q}}{1+|q|^2} \right)   \left( q'' -q'  \frac{\vp''}{\vp'} \right)   \\ 
&+ \frac{ 2q'^{\,2} \big(1-|q|^2 +2q^2|q|^2 \big) }{(1-q^2)^2(1+|q|^2)} -4 \left( \frac{q'\overline{q}}{1+|q|^2} \right)^2.  \nonumber
\end{align}
We proceed with the main theorem of this section. We denote by $T$ the upper triangle in $[0,1]^2$, that is,  
$$ 
T \, = \, \{ (s,t) \in[0,1]^2 \, : \, t\geq s\}. 
$$

\begin{theorem} \label{vp-thm} 
There exist non-negative functions $\eta$ and $c$ on $T$ such that if $\vp$ and $q$ are analytic in $\SD$, $\vp'\neq0$ and 
$$
|S\vp(z)| \, \leq \,  \frac{2 \, s}{(1-|z|^2)^2}, \quad z\in\SD
$$
for some $s\in[0,1]$ and such that $\om=q^2$ satisfies either 
$$
|\om(z)| \, \leq \,  \eta(s,t), \qquad z\in\SD,   
$$
or 
$$
0 \, < \,  1 -|\om(z)| \, \leq \,  c(s,t) |1 -\om(z)|, \qquad z\in\SD,    
$$
for some $t\in[s,1]$ then the lift $\widetilde f$ of the mapping $f=h+\overline{g}$ with $h-g=\vp$ and dilatation $\om=g'/h'$ satisfies  
$$
\Phi_f(z) \, \leq \, \frac{2\, t}{(1-|z|^2)^2}, \qquad z\in\SD. 
$$
In particular, $\widetilde f$ is univalent. The functions $\eta$ and $c$ may be chosen as
$$
\eta(s,t) \, = \, \frac{t-s}{7 + 4\sqrt{1+s} } \qquad \text{and} \qquad  c(s,t) \, = \, \frac{3(t-s)}{4(4 + 3\sqrt{1+s})}, \qquad  (s,t)\in T.
$$
\end{theorem}

The above functions $\eta(s,t)$ and $c(s,t)$ both attain their maximum on $T$ at the point $(0,1)$, the values being 1/11 and 3/28, respectively. Both decrease towards the diagonal $t=s$, where they vanish. 

The region of $w\in\SD$ for which $1-|w|\leq c \, |1-w|$ is the complement of a balloon-like set which is symmetric with respect to the horizontal axis and has an opening of $2\arccos(c)$ at its vertex $w=1$.

\begin{proof}[Proof of Theorem~\ref{vp-thm}]
We assume first that $|q| \leq r $ in $\SD$, for $r\in(0,1)$. The number $\eta$ will be the largest $R=r^2$ for which we will be able to infer that inequality \eqref{lift-t-thm} holds. We use formulas \eqref{def-Phi}, \eqref{curv} and \eqref{Schw-vp}, but also add and subtract the term $\frac{2 \, \bar{z} \, q'}{1-|z|^2}$, to compute
\begin{align*}
\Phi_f \,  \leq  & \, |S\vp|   + \frac{ 2 |q'|^2 (1-|q|^2 +2|q|^4) }{(1-|q|^2)^2(1+|q|^2)} + \frac{4|q'|^2}{1+|q|^2}  \\
 &+  2|q|  \left( \frac{1}{1-|q|^2} + \frac{1}{1+|q|^2} \right)  \left( \left| q'' -\frac{2 \, \bar{z} \, q'}{1-|z|^2} \right| + |q'|\left| \frac{\vp''}{\vp'}  -\frac{2 \, \bar{z}}{1-|z|^2} \right|\right)
\end{align*}
With the aid of the Schwarz-Pick lemma \eqref{Sch-Pick} and Lemmas~\ref{h} and~\ref{Wiener} applied to $q/r$, we compute 
\begin{align*}
(1 -|z|^2)^2 \Phi_f(z) \, \leq &\,  2s + \frac{2(r^2-|q(z)|^2)^2(1-|q(z)|^2 +2|q(z)|^4)}{r^2(1-|q(z)|^2)^2(1+|q(z)|^2)} \\
& + \frac{4(r^2-|q(z)|^2)^2}{r^2(1+|q(z)|^2)}  + \frac{8|q|(r^2-|q(z)|^2)}{r(1-|q(z)|^4)}(1+\sqrt{1+s}). 
 \end{align*} 
We write again $|\om|=|q|^2$ and see that in the second term 
$$
\frac{1-|\om| +2|\om|^2}{1+|\om|} \, = \,  2|\om| - 3 +\frac{4}{1+|\om|}  \, \leq \, 1
$$
since this function is convex in $|\om|$ and takes the value 1 at both endpoints of $[0,1]$. Hence the above is less than or equal to 
$$
2s + \frac{2(R-|\om|)^2}{R(1-|\om|)^2} +  \frac{4(R-|\om|)^2}{R(1+|\om|)} + 8(1+\sqrt{1+s}) \frac{(R-|\om|)}{1-|\om|^2}. 
$$
This can easily be seen to decrease in $|\om|\in[0,R]$, therefore it is smaller than 
$$
2s +2(7+4\sqrt{1+s})R, 
$$ 
which is less than or equal to $2t$ if and only if 
$$
R \, \leq \, \frac{t-s}{7 + 4\sqrt{1+s} }. 
$$
We take the right-hand side of this to be the function $\eta(s,t)$. 

\vskip.2cm
We now assume that $0<1-|\om| \leq c \,|1- \om| $ in $\SD$, and easily deduce from it that $|\om| \geq \frac{1-c}{1+c}$. We use \eqref{def-Phi}, \eqref{curv} and \eqref{Schw-vp}, and once again add and subtract the term $\frac{2 \, \bar{z} \, q'}{1-|z|^2}$, to compute
\begin{align*}
\Phi_f \,  \leq & \, |S\vp|   + \frac{ 2 c^2|q'|^2 (1-|q|^2 +2|q|^4) }{(1-|q|^2)^2(1+|q|^2)} + \frac{4|q'|^2}{1+|q|^2}  \\
&+  2|q|  \left( \frac{c}{1-|q|^2} + \frac{1}{1+|q|^2} \right)  \left( \left| q'' -\frac{2 \, \bar{z} \, q'}{1-|z|^2} \right| + |q'|\left| \frac{\vp''}{\vp'}  -\frac{2 \, \bar{z}}{1-|z|^2} \right|\right)
\end{align*}
Using the Schwarz-Pick lemma and Lemmas~\ref{h} and~\ref{Wiener}, we compute 
\begin{align*}
(1 -|z|^2)^2 \Phi_f(z) \,  \leq \, &   2s + \frac{2 c^2(1-|q(z)|^2 +2|q(z)|^4)}{1+|q(z)|^2} + \frac{4(1-|q(z)|^2)^2}{1+|q(z)|^2}  \\ 
& + 4 \,  \left( c +  \frac{1-|q(z)|^2}{1+|q(z)|^2}  \right)   \left( 1+ \sqrt{1+s} \right). 
  \end{align*} 
We write again $|\om|=|q|^2$ and as before, by convexity, we see that $\frac{1-|\om| +2|\om|^2}{1+|\om|}\leq 1$. Hence the above is smaller than
$$
2 s  +2c^2+ \frac{4(1-|\om|)^2}{1+|\om|}  + 4   \left(  c + \frac{1-|\om|}{1+|\om|}  \right) \left( 1+ \sqrt{1+s} \right),
$$
which, in view of $|\om| \geq \frac{1-c}{1+c}$, is smaller than 
$$
2s  + 2 c^2 + \frac{8c^2}{1+c} +8c\left( 1+ \sqrt{1+s} \right).
$$
Suppose now that $c\leq1/3$ in order to use the inequality $\frac{8c}{1+c} \leq 2$. We have that 
$$
(1-|z|^2)^2  \Phi_f(z)  \, \leq \, 2s  + 2 c \left( \tfrac{1}{3} + 5+ 4\sqrt{1+s}  \right),
$$ 
which is less than or equal to $2t$ if and only if 
$$
c \, \leq \, \frac{3(t-s)}{4(4 + 3\sqrt{1+s}) }.
$$
We may take this is as the function $c(s,t)$ since it is smaller than $1/3$ for $0 \leq s \leq t \leq1$.  
\end{proof}
 
The number $c=c(s,t)$ can be slightly improved to $c^*$ if we suppose that it is bounded by a parameter, say $\al$, (instead of $\frac{1}{3}$) and then find the optimal $\al$. This would improve its maximum to $c^*(0,1)\approx 0.1171$, which should be compared with our current $c(0,1)=3/28 \approx 0.1071$. Still, we do not know if this improved function is sharp in any sense. 

The simple transformation of conjugation $\widehat{f} = \overline{f} =\overline{h} +g $ can be used to generalize the hypothesis of Theorem~\ref{vp-thm}. It is evident that $\widehat\vp = -\vp$, so that the hypothesis on the Schwarzian derivative remains the same, while the dilatation is transformed into the meromorphic function $\widehat\om = 1/\om$. The corresponding sufficient conditions for $\widehat\om$ are
$$
|\widehat\om(z)| > \frac{1}{\eta(s,t)}  \qquad \text{or} \qquad 0<|\widehat\om(z)| - 1 \leq c (s,t) |1-\widehat\om(z)|, \qquad z\in\SD. 
$$
An analogous treatment of Theorem~\ref{h-thm} would be less successful since, in view of $\widehat{h} = g$, it would additionally change the hypothesis on the Schwarzian derivative. 

The following proposition justifies the need for some conditions on the dilatation such as the ones assumed in Theorem~\ref{vp-thm}. Numerically, it shows that if a function $\eta(s,t)$ gives a sufficient condition in Theorem~\ref{vp-thm} then it can not exceed $(t-s)/3$.

\begin{proposition}
For every $s$ and $t$ in $[0,1]$, such that $s\leq t$, and every $R>\frac{t-s}{3}$ there exists a harmonic mapping $f=h+\overline{g}$ for which, with the notation $h-g=\vp$, it holds that $\vp'\neq0$ in $\SD$, 
$$
 \|S\vp\| = 2s , \qquad \om(z)=R z^2
$$
and
$$
 \Phi_f(0) \, = \, 2s +2R \, > \,  2t. 
$$
\end{proposition}
\begin{proof}
As in the proof of Proposition~\ref{h-prop} we consider 
$$
\vp(z) = \left( \frac{1+z}{1-z} \right)^\al, \quad \al = \sqrt{1-s}, \quad \text{thus} \quad S\vp(z) = \frac{2s}{(1-z^2)^2}. 
$$
The computation follows in the same fashion in view of formulas \eqref{curv} and \eqref{Schw-vp}. 

\end{proof}

We now  prove Corollary~\ref{cor-vp}. 

\begin{proof}[Proof of Corollary~\ref{cor-vp}.]
We will simply see that for $t\in\big[s,\hat t(s) \big)$ we have
$$
\eta(s,t) \, = \, \frac{t-s}{7 + 4\sqrt{1+s} } \, < \, \left( \frac{1-\sqrt{t}}{1+\sqrt{t}} \right)^2 
$$
and use Theorem~\ref{plane-qc-ext} in order to deduce the univalence and quasiconformal extension of $f$. Denote the right-hand side of this inequality by $\psi$ and see that it decreases with $t$. On the other hand, $\eta$ increases with $t$ and decreases with $s$. For $s=1$ both sides meet at $t=1$ where they vanish. Therefore, for $s<1$ they meet at a unique point $t$ in $(0,1)$ and this is our number $\hat t(s)$. Again, since $\eta$ decreases with $s$, the function $\hat t$ increases from $\hat t(0) \, \approx \, 0.4431$ to $\hat t(1)=1$. 
\end{proof}

In order to prove Theorem~\ref{converse} we will make use of the following theorem from \cite{ABG}. See also \cite{CH} for a similar application. 
\begin{theoremO}[\cite{ABG}] \label{ABG}
Let $\Om\subsetneq \SC$ be a simply connected domain. Then $\Om$ satisfies the chord-arc condition \eqref{def-M} if and only if there exists $c>0$ such that the condition $|\Psi'(w)-1|<c$ for all $w\in\Om$ implies that the analytic map $\Psi:\Om\to\SC$ is injective. 
\end{theoremO}

\begin{proof}[Proof of Theorem~\ref{converse}]
In view of Theorem \ref{ABG}, for every $c>0$ there exists a non-injective holomorphic function $\Psi:\Om\to\SC$ that satisfies $|\Psi'(w)-1|<c$ in $\Om$. Let $w_1$ and $w_2$ be distinct points in $\Om$ for which $\Psi(w_1) = \Psi(w_2)$. Rotating the domain $\Om$ we may assume that 
$$
{\rm Im}\,w_1 \, = \, {\rm Im}\,w_2 \,.
$$
We will show that there exist non-univalent shears of $\vp$ in the horizontal direction with arbitrary small dilatation. The asserted in the statement direction $\la\in\ST$ results from reversing this rotation of $\Om$. 


We set $\Psi(w) = w + \psi(w)$, so that $|\psi'(w)|<c$. We define 
\be \label{def-g}
g'(z) \, = \, \tfrac{1}{2} \vp'(z) \psi'\big(\vp(z)\big), \qquad z\in\SD, 
\ee
with $g(0)=0$, and consider the mapping $f=h+\bar{g}$ with $h-g = \vp$. We compute its dilatation 
$$
\om(z) \, = \, \frac{g'(z)}{h'(z)} \, = \, \frac{g'(z)}{\vp'(z) + g'(z)} \, = \, \frac{\psi'\big(\vp(z)\big)}{2+\psi'\big(\vp(z)\big)}
$$
and see that it can be arbitrarily small since
$$
|\om(z) | \, < \, \frac{c}{2-c} \, = \, \ve,
$$
if we make the choice $c=\frac{2\ve}{1+\ve}$.

We now show that $f$ is not univalent by showing that the harmonic mapping
$$
F(w) \, = \, f\big(\vp^{-1}(w)\big) \, = \, w + 2 {\rm Re}\,\{g\big(\vp^{-1}(w)\big) \}, \qquad w\in\Om, 
$$
is not univalent. It follows from \eqref{def-g} that
$$
(g\circ\vp^{-1})'(w) \, = \, \tfrac{1}{2}  \psi'(w). 
$$
Therefore, 
\begin{align*}
F(w_1)-F(w_2) \, = & \, w_1-w_2 + 2 {\rm Re}\left\{ g\big(\vp^{-1}(w_1)\big) - g\big(\vp^{-1}(w_2)\big) \right\} \\
 = & \,  w_1-w_2 + {\rm Re}\left\{  \psi(w_1) - \psi(w_2) \right\} \\
 = & \, {\rm Re}\left\{  \Psi(w_1) - \Psi(w_2) \right\} \\
 = & \, 0,
\end{align*}
which completes the proof. 
\end{proof}

\emph{Acknowledgements}. The first author was partially supported by Fondecyt Grant \#1150115. The second author was partially supported by the Thematic Research Network MTM2015-69323-REDT, MINECO, Spain.

\end{document}